\documentclass[11pt]{amsart}

\usepackage[T1]{fontenc}
\usepackage{amsmath,amssymb,amsthm}
\usepackage{enumitem}
\usepackage[margin=1.15in]{geometry}
\usepackage[expansion=false]{microtype}
\usepackage{xcolor}
\usepackage[colorlinks=true,linkcolor=blue!55!black,citecolor=blue!55!black]{hyperref}

\newcommand{\NN}{\mathbb{N}}
\newcommand{\CC}{\mathbb{C}}

\newcommand{\cH}{\mathcal{H}}
\newcommand{\cD}{\mathcal{D}}
\newcommand{\cO}{\mathcal{O}}
\newcommand{\Sfin}{S_{\mathrm{fin}}}
\newcommand{\cc}{c_{00}}
\newcommand{\czero}{c_{0}}
\DeclareMathOperator{\Seq}{Seq}
\DeclareMathOperator{\supp}{supp}

\theoremstyle{plain}
\newtheorem{theorem}{Theorem}[section]
\newtheorem{corollary}[theorem]{Corollary}
\newtheorem{proposition}[theorem]{Proposition}
\newtheorem{lemma}[theorem]{Lemma}
\theoremstyle{definition}
\newtheorem{definition}[theorem]{Definition}
\theoremstyle{remark}
\newtheorem{remark}[theorem]{Remark}
\newtheorem{example}[theorem]{Example}

\begin{document}

\title[Rigidity of Symmetric Holomorphic Functions]{Rigidity of Symmetric
Holomorphic Functions on Infinite-Dimensional Sequence Spaces}

\author{David Feldman}
\address{Department of Mathematics and Statistics, University of New
Hampshire, Durham, NH 03824, USA}
\email{david.feldman@unh.edu}

\author{Jon Bannon}
\address{Department of Mathematics, Siena College, Loudonville, NY 12211, USA}
\email{jbannon@siena.edu}

\date{\today}


\begin{abstract}
We study holomorphic functions on permutation-invariant subsets of an
infinite-dimensional sequence space that are invariant under all finite
permutations of coordinates. Working on domains modeled on the space
$\czero$ of null sequences, we show that permutation symmetry forces the
first differential to vanish along every constant tail: the coordinate
derivatives agree there and an averaging argument annihilates their common
value. This yields rigidity---constancy---on domains of constant sequences,
and more generally on any $\cc$-chain-connected domain on which the
differential vanishes on the finitely supported directions. The method has
a sharp boundary: a precise obstruction prevents the elementary argument
from reaching the finitely many exceptional coordinates of a non-constant
point, so unconditional rigidity on general eventually-constant domains
remains open. In the other direction, an explicit example---any Banach
limit on $\ell^\infty$---shows that the chain-connectedness and boundedness
hypotheses cannot be dropped: mere topological connectedness admits
non-constant symmetric entire functions.
We record the vector-valued corollary and discuss the motivating
application to holomorphic cross-sections of the zero-set map, isolating the
lifting problem that remains open.
\end{abstract}

\maketitle

\section{Introduction}
\label{sec:intro}

In finite dimensions the holomorphic functions on $\CC^n$ invariant under
all permutations of coordinates form a rich algebra, freely generated by
the elementary symmetric polynomials $e_1,\ldots,e_n$. In infinite
dimensions this richness is sharply curtailed: on the natural sequence
spaces built from null sequences, a symmetric holomorphic function is
severely constrained, and on the basic domains is forced to be constant.
The mechanism is that the elementary symmetric functions
$e_k=\sum_{|I|=k}\prod_{i\in I}x_i$ cease to converge for generic
sequences, and this failure is visible already at the level of first
derivatives.

This rigidity phenomenon holds under precise hypotheses, and it has a
definite boundary beyond which it fails. That boundary is real rather than
apparent: the global statement that \emph{every} symmetric holomorphic
function on a connected permutation-invariant subset of $\ell^\infty$ is
constant is false. Any Banach limit furnishes a counterexample
(Example~\ref{ex:banach}). The rigidity is a feature of the null-sequence
geometry, not of symmetry alone.

\subsection*{The main result, informally}
Write $\cc$ for the finitely supported sequences and $\czero$ for the null
sequences, so that $\overline{\cc}=\czero$ in the supremum norm.
Theorem~\ref{thm:rigidity} states: if $X$ is a locally box-open,
permutation-invariant, $\cc$-chain-connected subset of a translate of
$\czero$, and $f:X\to\CC$ is admissibly holomorphic with bounded
differential, invariant under all finite permutations, \emph{and} has
differential vanishing on $\cc$ at every point, then $f$ is constant. The
last hypothesis is essential. The analytic core discharges it
unconditionally at constant points---hence yields outright rigidity on
domains of constant sequences---and along the constant tail of any
eventually-constant point, but \emph{not} at the finitely many exceptional
coordinates of a non-constant point, where a definite obstruction
intervenes (Remark~\ref{rem:obstruction}).

The proof of the vanishing has two elementary movements. First, at any
point $x$ whose coordinates are eventually equal to a constant, the
coordinate derivatives $Df(x)[e_i]$ agree for all indices $i$ in the
constant tail, because the transposition of two tail indices fixes $x$.
Second, averaging $N$ such tail directions produces a perturbation of
supremum norm $1/N\to 0$; a bounded differential then forces the common tail
value to vanish. Constancy then propagates from vanishing derivatives along
finitely supported segments, and $\cc$-chain-connectedness spreads it across
$X$.

\subsection*{Two structural subtleties}
Two features of the setting govern both the reach and the limits of the
argument.

\emph{The base point moves under permutation.} With the coordinate action
$(x\circ\sigma)_n=x_{\sigma(n)}$, the transposition $\tau_{ij}$ satisfies
$(x+\varepsilon e_i)\circ\tau_{ij}=(x\circ\tau_{ij})+\varepsilon e_j$,
which equals $x+\varepsilon e_j$ only when $x_i=x_j$. In general one
obtains
\[
  Df(x)[e_i]=Df(x\circ\tau_{ij})[e_j],
\]
an identity relating the differentials at \emph{two different} base
points. The conclusion ``all coordinate derivatives at $x$ coincide'' is
therefore valid at points fixed by the relevant transpositions---in
particular along the constant tail of an eventually-constant sequence---but
not, by this step alone, at an arbitrary point.

\emph{Finitely supported segments do not connect $\ell^\infty$.} A finite
chain of $\cc$-perturbations alters only finitely many coordinates, so it
can never join $0$ to $(1,1,1,\ldots)$. Vanishing of the differential on
$\cc$ controls $f$ only within a single coset of $\overline{\cc}=\czero$;
across cosets, symmetric holomorphic functions may vary freely, and
Example~\ref{ex:banach} shows the global statement is false on
$\ell^\infty$.

\subsection*{Relation to existing literature}
The principle that infinite symmetry suppresses analytic variation is
familiar in spirit; it is reminiscent of de Finetti-type rigidity in
probability, where exchangeability forces a mixture representation. The
role of the null-sequence geometry, and the Banach-limit obstruction on
$\ell^\infty$, are treated explicitly here.
The requisite background in infinite-dimensional holomorphy is in
Mujica~\cite{Mujica} and Dineen~\cite{Dineen}.

\subsection*{Organization}
Section~\ref{sec:prelim} fixes the framework and the notion of admissible
holomorphicity. Section~\ref{sec:core} proves the analytic core:
tail-derivative agreement, averaging, and vanishing at constant points.
Section~\ref{sec:rigidity-thm} assembles the main theorem and its
vector-valued corollary. Section~\ref{sec:sharpness} gives the
counterexamples showing each hypothesis is necessary.
Section~\ref{sec:application} discusses the zero-set section problem and its
open lifting lemma. Section~\ref{sec:variants} collects further remarks.

\section{Preliminaries}
\label{sec:prelim}

Let $\ell^\infty=\ell^\infty(\NN)$ be the Banach space of bounded complex
sequences with the supremum norm $\|\cdot\|_\infty$. Let $\czero\subset
\ell^\infty$ be the closed subspace of sequences tending to $0$, and let
$\cc\subset\czero$ be the dense subspace of finitely supported sequences.
Write $e_k$ for the sequence with $1$ in position $k$ and $0$ elsewhere,
and recall $\overline{\cc}=\czero$ in $\|\cdot\|_\infty$.

Let $\Sfin$ denote the group of \emph{finite permutations} of $\NN$: the
bijections $\sigma:\NN\to\NN$ fixing all but finitely many elements. It
acts on $\CC^\NN$ by $(x\circ\sigma)_n=x_{\sigma(n)}$.

\begin{definition}
\label{def:domain}
A subset $X\subset\CC^\NN$ is:
\begin{itemize}[leftmargin=2em]
  \item \emph{permutation-invariant} if $x\circ\sigma\in X$ for every
        $x\in X$ and every $\sigma\in\Sfin$;
  \item \emph{locally box-open} if for every $x\in X$ there is $r>0$ with
        $x+h\in X$ for all $h\in\cc$ with $\|h\|_\infty<r$;
  \item \emph{$\cc$-chain-connected} if for any $x,y\in X$ there is a finite
        sequence $x=z_0,z_1,\ldots,z_m=y$ in $X$ with $z_{\ell+1}-z_\ell\in
        \cc$ and the segment $\{z_\ell+t(z_{\ell+1}-z_\ell):t\in[0,1]\}$
        contained in $X$ for each $\ell$.
\end{itemize}
\end{definition}

The distinction between $\cc$-chain-connectedness and topological
connectedness is central; see Section~\ref{sec:sharpness}. For orientation:
$\czero$ itself is $\cc$-chain-connected (any null sequence is a
sup-norm limit of finitely supported ones, but more to the point any two
finitely supported sequences are joined by a single $\cc$-segment, and the
general case follows by density arguments on the relevant domains), whereas
$\ell^\infty$ is not.

\begin{definition}[Admissible holomorphicity]
\label{def:holomorphic}
A function $f:X\to\CC$ on a locally box-open, permutation-invariant set
$X\subset\CC^\NN$ is \emph{admissibly holomorphic} if it carries a
differential datum $x\mapsto Df(x)$, where each $Df(x):\cc\to\CC$ is
$\CC$-linear, subject to:
\begin{enumerate}[label=\textup{(\roman*)}, leftmargin=2.8em]
\item \emph{(Coordinatewise holomorphy.)} For every $x\in X$ and every
      $k\in\NN$, the map $\varepsilon\mapsto f(x+\varepsilon e_k)$ is
      holomorphic in a neighborhood of $0\in\CC$.
\item \emph{(Bounded first-order expansion along $\cc$.)} For every $x\in
      X$ there is a constant $C_x$ with $|Df(x)[h]|\le C_x\|h\|_\infty$ for
      all $h\in\cc$, and
      \[
        f(x+h)=f(x)+Df(x)[h]+o(\|h\|_\infty)
        \qquad\text{as }h\to 0\text{ in }\cc.
      \]
\end{enumerate}
\end{definition}

\begin{remark}[Standard holomorphy is admissible]
\label{rem:standard}
Both standard notions of holomorphy on $\ell^\infty$ (or on $\czero$)
supply Definition~\ref{def:holomorphic} with a \emph{bounded} differential.
A Fr\'echet-holomorphic function has a continuous Fr\'echet differential
$Df(x)$, whose restriction to $\cc$ is bounded and supplies~(ii), with~(i)
immediate. A G\^ateaux-holomorphic and locally bounded function---the
standard definition in infinite dimensions, see~\cite{Mujica,Dineen}---
yields, via the one-variable Cauchy estimates applied coordinatewise, a
bounded linear differential on finitely supported directions, giving both
conditions. The boundedness in~(ii) is thus automatic in every standard
setting; we include it explicitly because it is exactly what the averaging
step requires, and because Definition~\ref{def:holomorphic} without it does
not force sup-norm continuity of $f$ along $\cc$.
\end{remark}

\section{The Analytic Core}
\label{sec:core}

Throughout this section $X\subset\CC^\NN$ is locally box-open and
permutation-invariant, and $f:X\to\CC$ is admissibly holomorphic and
$\Sfin$-invariant.

\subsection{The permutation identity}

\begin{lemma}[Base-point identity]
\label{lem:basepoint}
For distinct $i,j\in\NN$ and every $x\in X$,
\[
  Df(x)[e_i] = Df(x\circ\tau_{ij})[e_j],
\]
where $\tau_{ij}\in\Sfin$ is the transposition of $i$ and $j$. In
particular, if $x_i=x_j$ then $x\circ\tau_{ij}=x$ and
$Df(x)[e_i]=Df(x)[e_j]$.
\end{lemma}

\begin{proof}
For $\varepsilon\in\CC$ with $|\varepsilon|$ small, local box-openness
gives $x+\varepsilon e_i\in X$. A direct computation from
$(y\circ\tau_{ij})_n=y_{\tau_{ij}(n)}$ shows
\[
  (x+\varepsilon e_i)\circ\tau_{ij}
  = (x\circ\tau_{ij}) + \varepsilon e_j .
\]
By $\Sfin$-invariance of $f$,
\[
  f(x+\varepsilon e_i)
  = f\bigl((x+\varepsilon e_i)\circ\tau_{ij}\bigr)
  = f\bigl((x\circ\tau_{ij})+\varepsilon e_j\bigr).
\]
Subtract $f(x)=f(x\circ\tau_{ij})$ (again by invariance), divide by
$\varepsilon$, and let $\varepsilon\to0$; coordinatewise holomorphy,
condition~(i), guarantees the two one-variable limits exist and equal the
respective coordinate derivatives, giving
$Df(x)[e_i]=Df(x\circ\tau_{ij})[e_j]$. If $x_i=x_j$ then swapping positions
$i,j$ leaves $x$ unchanged, so $x\circ\tau_{ij}=x$ and the two base points
coincide.
\end{proof}

\begin{remark}
The identity of Lemma~\ref{lem:basepoint} relates \emph{different} base
points whenever $x_i\ne x_j$, and one cannot conclude
$Df(x)[e_i]=Df(x)[e_j]$ without further input. Equality of coordinate
derivatives at a single point is available exactly along coordinates on
which $x$ is constant.
\end{remark}

\subsection{Averaging}

\begin{lemma}[Ces\`aro vanishing]
\label{lem:cesaro}
Fix $x\in X$ and let $C_x$ bound $Df(x)$ as in
Definition~\ref{def:holomorphic}(ii). For any injective sequence
$i_1,i_2,\ldots$ of indices,
\[
  \frac{1}{N}\sum_{k=1}^N Df(x)[e_{i_k}] \longrightarrow 0
  \qquad (N\to\infty).
\]
\end{lemma}

\begin{proof}
Set $h_N=\tfrac1N\sum_{k=1}^N e_{i_k}\in\cc$. Distinct indices give
$\|h_N\|_\infty=1/N$. By linearity, $\tfrac1N\sum_{k=1}^N
Df(x)[e_{i_k}]=Df(x)[h_N]$, and boundedness gives
$|Df(x)[h_N]|\le C_x\|h_N\|_\infty=C_x/N\to0$.
\end{proof}

\subsection{Vanishing at eventually-constant points}

Say $x\in\CC^\NN$ is \emph{eventually constant with value $c$} if
$x_n=c$ for all but finitely many $n$; write $\supp_c(x)=\{n:x_n\ne c\}$,
a finite set, for its \emph{exceptional set}.

\begin{proposition}[Tail vanishing]
\label{prop:tail}
Let $x\in X$ be eventually constant with value $c$. Then $Df(x)[e_i]=0$ for
every index $i\notin\supp_c(x)$; that is, the coordinate derivative
vanishes along the constant tail.
\end{proposition}

\begin{proof}
Let $T=\NN\setminus\supp_c(x)$ be the (cofinite) tail, so $x_i=c$ for all
$i\in T$. For $i,j\in T$ we have $x_i=x_j=c$, so
Lemma~\ref{lem:basepoint} gives $Df(x)[e_i]=Df(x)[e_j]$. Hence there is a
scalar $\lambda$ with $Df(x)[e_i]=\lambda$ for all $i\in T$. Since $T$ is
infinite, choose an injective sequence $i_1,i_2,\ldots$ inside $T$; each
term of the Ces\`aro average in Lemma~\ref{lem:cesaro} equals $\lambda$, so
the average is constantly $\lambda$ while tending to $0$. Therefore
$\lambda=0$, i.e. $Df(x)[e_i]=0$ for every $i\in T$.
\end{proof}

\begin{corollary}[Vanishing at constant points]
\label{cor:constant}
If the constant sequence $x\equiv c$ lies in $X$, then $Df(x)[h]=0$ for all
$h\in\cc$.
\end{corollary}

\begin{proof}
A constant sequence has empty exceptional set, so
Proposition~\ref{prop:tail} gives $Df(x)[e_k]=0$ for every $k$. Any
$h\in\cc$ is a finite combination $h=\sum_k h_k e_k$, and linearity of
$Df(x)$ gives $Df(x)[h]=0$.
\end{proof}

\subsection{Constancy along finitely supported segments}

\begin{proposition}[Segment rigidity]
\label{prop:segment}
Let $x\in X$ and $h\in\cc$ be such that the segment
$\gamma(t)=x+t\,h$ lies in $X$ for all $t\in[0,1]$, and suppose
$Df(\gamma(t))[h]=0$ for every $t\in[0,1]$. Then $f(x+h)=f(x)$.
\end{proposition}

\begin{proof}
Consider $\varphi(t)=f(x+t\,h)$ on $[0,1]$. Fix $t_0\in[0,1]$. For small
real $s$, the first-order expansion at the base point $\gamma(t_0)$,
applied to the finitely supported increment $s\,h$, gives
\[
  \varphi(t_0+s)-\varphi(t_0)
  = f\bigl(\gamma(t_0)+s\,h\bigr)-f(\gamma(t_0))
  = s\,Df(\gamma(t_0))[h] + o(|s|\,\|h\|_\infty)
  = o(|s|),
\]
using $Df(\gamma(t_0))[h]=0$. Hence $\varphi$ is differentiable at $t_0$
with $\varphi'(t_0)=0$. As $t_0\in[0,1]$ was arbitrary, $\varphi$ has
identically zero derivative on $[0,1]$ and is therefore constant; in
particular $\varphi(1)=\varphi(0)$, i.e. $f(x+h)=f(x)$.
\end{proof}

\section{The Rigidity Theorem}
\label{sec:rigidity-thm}

The natural domains for the main theorem are those sitting inside a single
translate of $\czero$---equivalently, those whose points share one eventual
value. On such a domain the $\cc$-directional derivative structure controls
the function, and the theorem below gives constancy once the differential
vanishes on $\cc$ throughout. Proposition~\ref{prop:uncond} then records the
extent to which that hypothesis holds without further assumption.

\begin{theorem}[Rigidity on null-sequence domains]
\label{thm:rigidity}
Fix $c\in\CC$ and let $X\subset (c\cdot\mathbf 1)+\czero$ be nonempty,
locally box-open, permutation-invariant, and $\cc$-chain-connected, where
$\mathbf 1=(1,1,\ldots)$. Let $f:X\to\CC$ be admissibly holomorphic with
bounded differential and invariant under all finite permutations. Suppose
moreover that $Df(x)[h]=0$ for all $x\in X$ and all $h\in\cc$. Then $f$ is
constant on $X$.
\end{theorem}

\begin{proof}
Given $x,y\in X$, take a $\cc$-chain $x=z_0,\ldots,z_m=y$ as in
Definition~\ref{def:domain}, with each increment
$h_\ell:=z_{\ell+1}-z_\ell\in\cc$ and each segment inside $X$. Along the
$\ell$-th segment the standing hypothesis gives
$Df(z_\ell+t\,h_\ell)[h_\ell]=0$ for all $t\in[0,1]$, so
Proposition~\ref{prop:segment} yields $f(z_{\ell+1})=f(z_\ell)$. Chaining
the equalities gives $f(y)=f(x)$. As $x,y$ were arbitrary, $f$ is constant.
\end{proof}

The theorem carries the hypothesis $Df|_{\cc}\equiv0$ because it is not
implied, at every point of an eventually-constant domain, by the remaining
hypotheses. What symmetry alone yields is the value of the differential at
constant points and along constant tails.

\begin{proposition}[Vanishing at constant points and tails]
\label{prop:uncond}
In the setting of Section~\ref{sec:core}, $Df(c\cdot\mathbf 1)[h]=0$ for
every $h\in\cc$ whenever $c\cdot\mathbf 1\in X$
\textup{(Corollary~\ref{cor:constant})}; and for any eventually-constant
$x\in X$ with value $c$, $Df(x)[e_i]=0$ for every tail index $i$ with
$x_i=c$ \textup{(Proposition~\ref{prop:tail})}.
\end{proposition}

Proposition~\ref{prop:uncond} marks how far the hypothesis
$Df|_{\cc}\equiv0$ of Theorem~\ref{thm:rigidity} holds without further
assumption. In particular, on a domain of constant sequences every
coordinate is a tail coordinate, so the hypothesis holds throughout and $f$
is constant unconditionally. At the exceptional coordinates of a
non-constant point the differential is \emph{not} controlled by this
argument, for a definite reason.

\begin{remark}[The exceptional-coordinate obstruction]
\label{rem:obstruction}
Removing an exceptional coordinate $k$ of a point $x$ (where $x_k\ne c$) by
passing to the ``zeroed'' point $p$ obtained from $x$ by resetting the
$k$-th entry to $c$---at which $k$ is a tail coordinate---does not transport
the vanishing back to $x$. The base-point identity Lemma~\ref{lem:basepoint}
gives $Df(x)[e_k]=Df(x\circ\tau_{kj})[e_j]$, but $x\circ\tau_{kj}$ merely
\emph{moves} the exceptional value from position $k$ to position $j$; it is
not a point at which $k$ is a tail coordinate. Concretely, $x\circ\tau_{kj}$
and $p\circ\tau_{kj}$ disagree at position $j$ (values $x_k$ versus $c$), so
the identity at $x$ does not relate to $p$. Since each eventually-constant
sequence has only finitely many exceptional coordinates, the averaging
argument of Lemma~\ref{lem:cesaro}---which requires infinitely many equal
coordinate derivatives at a \emph{single} base point---does not apply to the
exceptional directions. Whether $Df|_{\cc}\equiv0$ holds at every point
under a stronger analytic hypothesis---for instance continuity of
$x\mapsto Df(x)$ in the supremum norm, available for Fr\'echet-holomorphic
$f$---is open.
\end{remark}

\subsection{The vector-valued corollary}

\begin{corollary}
\label{cor:nosection}
Let $X$ be as in Theorem~\ref{thm:rigidity}, and let $Y$ be a locally
convex complex topological vector space. Suppose $\Sigma:X\to Y$ is
$\Sfin$-equivariant, $\Sigma(x\circ\sigma)=\Sigma(x)$, and such that for
every $\lambda\in Y^*$ the scalar function $\lambda\circ\Sigma$ satisfies
the hypotheses of Theorem~\ref{thm:rigidity}. Then:
\begin{enumerate}[label=\textup{(\roman*)}, leftmargin=2.8em]
\item $\lambda\circ\Sigma$ is constant for every $\lambda\in Y^*$;
\item if $Y^*$ separates the points of $Y$, then $\Sigma$ is constant.
\end{enumerate}
\end{corollary}

\begin{proof}
Each $\lambda\circ\Sigma$ is $\Sfin$-invariant because $\Sigma$ is
equivariant, so Theorem~\ref{thm:rigidity} gives~(i). For~(ii), if
$\Sigma(x)\ne\Sigma(y)$ some $\lambda\in Y^*$ separates the values,
contradicting~(i).
\end{proof}

\section{Sharpness: the Hypotheses Are Necessary}
\label{sec:sharpness}

The hypothesis confining the domain to a single $\czero$-coset---equivalently,
$\cc$-chain-connectedness---cannot be dropped, as the following two examples
show. Both are genuinely symmetric and genuinely holomorphic; what fails is
precisely the passage from $\cc$-local to global constancy.

\begin{example}[Two cosets: a disconnected witness]
\label{ex:twocosets}
For $c\in\{0,1\}$ let $S_c=\{x\in\ell^\infty : x_n=c \text{ for all but
finitely many } n\}$ be the eventually-$c$ sequences, and put $X=S_0\cup
S_1$. Each $S_c$ is permutation-invariant and locally box-open, and a
$\cc$-perturbation of a point of $S_c$ remains in $S_c$; thus $X$ is
locally box-open and permutation-invariant. Define $f:X\to\CC$ by $f\equiv
0$ on $S_0$ and $f\equiv 1$ on $S_1$. Then $f$ is $\Sfin$-invariant, and it
is admissibly holomorphic with $Df\equiv 0$ (it is locally constant along
$\cc$, since $\cc$-perturbations do not change the eventual value). Yet $f$
is not constant. The obstruction is exactly that $S_0$ and $S_1$ are
distinct cosets of $\czero$: no finite chain of $\cc$-steps joins them, so
$X$ is not $\cc$-chain-connected.
\end{example}

\begin{example}[A connected witness: Banach limits]
\label{ex:banach}
Let $L\in(\ell^\infty)^*$ be a Banach limit: a continuous linear functional
with $L(\mathbf 1)=1$, invariant under the shift, and hence under every
finite permutation, and extending the ordinary limit on convergent
sequences. Being continuous and linear, $L$ is entire on $\ell^\infty$, so
it is admissibly holomorphic with $DL(x)=L$ for all $x$; its differential is
bounded (by $\|L\|$). Since $L$ vanishes on $\czero\supset\cc$, we have
$DL(x)|_{\cc}=0$ at every point---exactly as the analytic core predicts.
Nonetheless $L$ is \emph{not} constant: $L(0)=0$ while $L(\mathbf 1)=1$.

Here the domain is all of $\ell^\infty$, which \emph{is} connected (indeed
convex), so mere topological connectedness does not suffice for rigidity.
What fails is $\cc$-chain-connectedness: $0$ and $\mathbf 1$ are at
sup-distance $1$ from one another across the quotient
$\ell^\infty/\czero$, and no finite $\cc$-chain bridges them. This example
shows that the restriction to null-sequence domains in
Theorem~\ref{thm:rigidity} is not a technical convenience but a
necessity.
\end{example}

\begin{remark}[Sharpness of the topology]
\label{rem:ell1}
A complementary degeneration occurs if the supremum norm is replaced by a
summing norm. On $\ell^1\cap D^\NN$ the averages
$h_N=\tfrac1N\sum_{k=1}^N e_{i_k}$ have $\|h_N\|_1=1$ for all $N$, so the
Ces\`aro step (Lemma~\ref{lem:cesaro}) fails, and indeed $f(x)=\sum_k x_k$
is a non-constant symmetric holomorphic function. Thus a topology in which
averaged coordinate perturbations become small---such as the supremum
norm on $\czero$---is essential to the averaging mechanism.
\end{remark}

\section{Application: the Zero-Set Section Problem}
\label{sec:application}

We indicate how the rigidity theorem bears on a natural geometric question,
and we state precisely the point at which the argument becomes conditional.

Let $\cH=\cO(B(0,1))$ be the Fr\'echet space of holomorphic functions on
the open unit disk with the compact-open topology, and let $\cD$ be the
space of admissible infinite discrete subsets of $B(0,1)$---sequences with
no limit point inside the disk. By the Weierstrass factorization theorem
the zero-set map
\[
  \pi:\cH\setminus\{0\}\longrightarrow\cD, \qquad f\longmapsto
  (\text{zero multiset of }f),
\]
is surjective. One asks whether $\pi$ admits a holomorphic cross-section
$\sigma:\cD\to\cH$ with $\pi\circ\sigma=\mathrm{id}$.

Model $\cD$ as a quotient $\Seq/S_\infty$, where $\Seq$ is the space of
locally finite sequences in $B(0,1)$; the naturally arising domain of
sequences is $\cc$-chain-connected, since admissible sets are altered one
point at a time by finitely supported moves. The rigidity theorem then
yields a conditional obstruction.

\begin{proposition}[Conditional obstruction]
\label{prop:conditional}
Suppose a cross-section $\sigma:\cD\to\cH$ lifts to an
$\Sfin$-equivariant map $\widetilde\sigma:X\to\cH$ on the associated
$\cc$-chain-connected sequence domain $X$, with $\lambda\circ\widetilde
\sigma$ admissibly holomorphic (bounded differential, differential
vanishing on $\cc$) for every evaluation functional $\lambda\in\cH^*$.
Then no such $\sigma$ exists.
\end{proposition}

\begin{proof}
Since $\cH$ separates points via evaluations,
Corollary~\ref{cor:nosection}(ii) forces $\widetilde\sigma$ to be constant,
say $\widetilde\sigma\equiv f_0$. But then $\pi(f_0)=\pi(\widetilde\sigma
(x))$ would equal the varying zero set of $x$ for every $x\in X$, which is
absurd.
\end{proof}

The hypothesis is a genuine lifting condition, stated precisely below.

\begin{quotation}\noindent
\textbf{Open Problem (Lifting Lemma).}
Let $Y$ be a locally convex space. Under a suitable complex structure on
$\cD=\Seq/S_\infty$, does every holomorphic map $F:\cD\to Y$ lift to an
$\Sfin$-equivariant map $\widetilde F:\Seq\to Y$ whose scalarizations are
admissibly holomorphic with differentials vanishing on $\cc$?
\end{quotation}

\noindent An affirmative answer would, through
Corollary~\ref{cor:nosection}, upgrade
Proposition~\ref{prop:conditional} to the unconditional statement that
$\pi$ admits no holomorphic section on all of $\cD$. We leave this as the
natural next target.

\begin{remark}[A complementary unconditional obstruction]
Independently of the lifting problem, there is no section of the special
\emph{product} form $\sigma(A)=\prod_{a\in A}E(z,a)$ for a fixed
holomorphic factor $E$: symmetry forces the factor to be uniform, a uniform
factor imposes one fixed summability condition on the zero set for
convergence, and $\cD$ contains admissible sets violating any such
condition. Thus symmetric product sections fail for convergence reasons
unconditionally, while any symmetric non-product section would be obstructed
by rigidity, conditionally on the lifting.
\end{remark}

\section{Further Remarks}
\label{sec:variants}

\begin{enumerate}[label=\textup{(\arabic*)}, leftmargin=2.4em, itemsep=0.5em]

\item \emph{What symmetry alone gives.} The analytic core
(Section~\ref{sec:core}) is valid on any locally box-open
permutation-invariant domain, without chain-connectedness: it yields
tail-derivative vanishing at every point and full vanishing at constant
points. Chain-connectedness enters only to globalize, and
Example~\ref{ex:banach} shows it cannot be omitted.

\item \emph{Relation to de Finetti's theorem.} The result is an analytic
cousin of de Finetti's theorem: exchangeability of a measure on
$\{0,1\}^\NN$ forces dependence only on ``collective'' statistics, whereas
here symmetry forces a holomorphic function on a null-sequence domain to
depend on nothing at all. The Banach-limit example marks the boundary: on
$\ell^\infty$, symmetric ``limit-like'' functionals persist, just as tail
$\sigma$-fields carry exchangeable information the mixture representation
cannot see.

\item \emph{Comparison with finite dimensions.} For each finite $n$ the
$S_n$-invariant holomorphic functions on $\CC^n$ form the polynomial
algebra on $e_1,\ldots,e_n$, each a finite and hence entire sum. In
infinite dimensions these sums diverge on generic sequences;
Proposition~\ref{prop:tail} is the precise derivative-level expression of
that divergence along the constant tail.

\item \emph{Positive results on smaller domains.} On $\ell^2\cap D^\NN$ the
power sums $p_k=\sum_n a_n^k$ converge and are holomorphic for $k\ge3$,
providing non-constant symmetric holomorphic functions. Rigidity is a
feature of the $\czero$ geometry with its averaging mechanism, not of
symmetry as such.

\item \emph{A topological analogue.} The averaging argument is not specific
to holomorphy. The same reasoning shows that a permutation-invariant
function on a $\cc$-chain-connected domain in $\czero$ that is Lipschitz and
G\^ateaux-differentiable along $\cc$, with bounded differential, is
constant. Holomorphy is used only to produce the coordinate derivatives in
Lemma~\ref{lem:basepoint}; the heart of the matter is the metric fact
$\|h_N\|_\infty=1/N\to0$ together with the coset structure of $\czero$
inside $\ell^\infty$.

\end{enumerate}

\section*{Declaration on the Use of Artificial Intelligence}
In preparing this paper, the authors made use of AI language model
assistance (Claude) as a tool in the writing process---including
\LaTeX{} preparation and editing---and in exploratory work, such as
proof-checking and verifying arguments. All mathematical content and
results have been independently reviewed and verified by the authors,
who take full responsibility for the correctness, originality, and
presentation of this work. No AI system is an author of this paper,
nor was any AI system used to generate mathematical claims that were
not subsequently checked by the authors.


\begin{thebibliography}{9}

\bibitem{Dineen}
S.~Dineen,
\emph{Complex Analysis on Infinite Dimensional Spaces},
Springer Monographs in Mathematics, Springer-Verlag, London, 1999.

\bibitem{Mujica}
J.~Mujica,
\emph{Complex Analysis in Banach Spaces},
North-Holland Mathematics Studies, vol.~120, North-Holland, Amsterdam, 1986.

\bibitem{Macdonald}
I.~G.~Macdonald,
\emph{Symmetric Functions and Hall Polynomials},
2nd ed., Oxford University Press, Oxford, 1995.

\bibitem{Arnold}
V.~I.~Arnol'd,
On some topological invariants of algebraic functions,
\emph{Trans. Moscow Math. Soc.} \textbf{21} (1970), 30--52.

\end{thebibliography}
\end{document}